\documentclass[12pt]{article}   
\usepackage{graphicx,psfrag}
\usepackage[table]{xcolor}
\usepackage{color}
\usepackage{tikz}
\usepackage{tkz-graph}
\usepackage{array}
\usepackage{colortbl}
\usepackage{amsmath,amsfonts}
\usepackage{amssymb,amsthm}

\usepackage{epstopdf}

\usepackage{tabularx}
\newtheorem{theorem}{Theorem}[section]
\newtheorem{corollary}[theorem]{Corollary}
\newtheorem{proposition}[theorem]{Proposition}

\newtheorem{example}[theorem]{Example}
\newtheorem{lemma}[theorem]{Lemma}

\newtheorem{question}[theorem]{Question}

\def\qed{\vbox{\hrule
 \hbox{\vrule\hbox to 5pt{\vbox to 8pt{\vfil}\hfil}\vrule}\hrule}}

\def\endproof{\unskip \nobreak \hskip0pt plus 1fill \qquad \qed \par}

\newcommand{\mb}{\mathbb}
\newcommand{\mc}{\mathcal}

\usepackage{tikz}
\usetikzlibrary{decorations.markings}

\begin{document}

\title{Permutation  Matrices, Their Discrete Derivatives and Extremal Properties}

\author{
 Richard A. Brualdi\footnote{Department of Mathematics, University of Wisconsin, Madison, WI 53706, USA. {\tt brualdi@math.wisc.edu}} \\
 Geir Dahl\footnote{Department of Mathematics,  
 University of Oslo, Norway.
 {\tt geird@math.uio.no.} Corresponding author.}
 }

\maketitle

\begin{center}
   {\em Dedicated to Volker Mehrmann with admiration and respect.}
\end{center}

\begin{abstract} 
  For a permutation $\pi$, and the corresponding permutation matrix, we introduce the notion of {\em discrete derivative}, obtained by taking differences of successive entries in $\pi$. We characterize the possible derivatives of permutations, and consider questions for permutations with certain properties satisfied by the derivative. For instance, we consider  permutations with distinct derivatives, and the relationship to  so-called Costas arrays.
 
 \end{abstract}




\section{Introduction}
 \label{sec:intro}
 Let $\pi=(\pi_1,\pi_2,\ldots,\pi_n)$ be a permutation of $\{1,2,\ldots,n\}$, a {\it permutation of order $n$}. The permutation $\pi$ can be given in the equivalent form as  an $n\times n$ permutation matrix $P_{\pi}$ with 1's in positions $(i,\pi_i)$ for $i=1,2,\ldots,n$ and 0's elsewhere. 
 Let $\mc{S}_n$ denote the set of all permutations of order $n$, and let $\mc{P}_n$ denote the corresponding set of $n\times n$ permutation matrices. 
We define the {\it discrete derivative } of $\pi$ and $P_{\pi}$ to be the vector
\[D(\pi)=D(P_{\pi}) = (\pi_2-\pi_1,\pi_3-\pi_2,\ldots,\pi_n-\pi_{n-1})\in \mb{Z}^{n-1}.\]
The integers $\pi_i-\pi_{i-1}$ are the {\it first-order differences} of $\pi$. For $k=1,2,\ldots,n-1$ we define the {\it $k$th order differences} of $\pi$ to be
integers $\pi_{i}-\pi_{i-k}$  for $i=k+1,\ldots,n$. Note that the $k$th order differences for $k\ge 2$ are sums of first-order differences:
\begin{equation}\label{eq:sums}\pi_i-\pi_{i-k}=(\pi_i-\pi_{i-1})+(\pi_{i-1}-\pi_{i-2})+\cdots+ (\pi_{i-k+1}-\pi_{i-k}).\end{equation}
If $k=0$, then the zeroth order differences are defined to be the entries of $\pi$ itself. We assemble all these differences in a triangle $\Theta_{\pi}$ that we call the {\it difference triangle} of $\pi$.
To illustrate, if $n=6$  and $\pi=(3,5,1,6,2,4)$, then
\begin{equation}\label{eq:difftable}
\Theta_{\pi}=\quad \begin{array}{rrrrrrrrrrr}
3&&5&&1&&6&&2&&4\\
&2&&-4&&5&&-4&&2&\\
&&-2&&1&&1&&-2&&\\
&&&3&&-3&&3&&&\\
&&&&-1&&-1&&&&\\
&&&&&1&&&&&
\end{array}.\end{equation}
We label the rows of the difference triangle of a permutation of order $n$ as $0,1,2,\ldots,$ $n-1$ so that they correspond to the order   of the differences in the rows.

The difference triangle $\Theta_{\pi}$ as defined above is different from what is usually called the  {\it difference table} of a finite sequence that is constructed by taking successive differences of entries in  rows. Rows 0 and 1 are the same but then they would differ. Thus, with $\pi=(3,5,1,6,2,4)$ as above, row 2 of the difference table is 
\[(-4)-2=-6\quad  5-(-4)=9\quad
 (-4)-5=-9\quad 2-(-4)=6,\]
 and differs from row 2 of the difference triangle as given in (\ref{eq:difftable}).

The notion of the derivative of a permutation captures the changes in consecutive entries of a permutation $\pi$, and therefore contains information about e.g. the descents of a permutation. We refer to the book \cite{Bona04} on permutations and their descents. Permutation matrices and more general classes of $(0,1)$-matrices are treated in \cite{BR91}. 

A {\it Costas permutation} (or {\it Costas array} or {\it Costas permutation matrix}) (\cite{Golomb}) is a permutation $\pi$ of order $n$ such that  for each $k=0,1,2,\ldots,n-1$, its  $k$th order differences in row $k$ of its difference triangle are distinct. Since $\pi$ is a permutation, its zeroth order differences are distinct; since $\pi$ has only one $(n-1)$st  difference, no restriction is placed on $(n-1)$st order differences. If we think of a permutation matrix  as a configuration of points in the Euclidean plane at the integral positions
$(i,\pi_i)$ for $i=1,\ldots,n$, then for a Costas permutation, no two of the  line segments determined by these points have both the same length and the same slope, and thus all the line segments they determine are distinct.
In terms of its difference triangle, the integers in each row are distinct. We remark that a {\it Golomb ruler} of order $n$  is defined to be a sequence of $n$ distinct positive integers such that {\it all} of the entries in its difference triangle are distinct. Since we confine our attention to  permutations of $\{1,2,\ldots,n\}$  a Golumb ruler is possible only in the trivial cases of  $n=1$ or $2$.
An example of a Costas permutation of order 4 is $\pi=(4,3,1,2)$ with difference triangle
\[\begin{array}{rrrrrrr}
4&&3&&1&&2\\
&-1&&-2&&1&\\
&&-3&&-1&&\\
&&&-2&&&\end{array}.\]
 If a symmetry of the dihedral group $D_4$ is applied to a Costas permutation, the result is also a Costas permutation. 
Other structural properties of Costas permutations are given in \cite{Jedwab}. In particular, the following holds:
\begin{proposition}\label{pr:Jed} {\rm (Proposition 4.2 in \cite{Jedwab})}.
If $n\ge 4$, then in a Costas permutation matrix  $P=[p_{ij}]$ of order $n$,
there exists $p_{rs}=1$ and $p_{uv}=1$ $((r,s)\ne (u,v))$ such that also
$p_{ab}=1$ and $p_{cd}=1$ where $b-d=s-v$ and $a-c=-(r-u)$, that is the line segment joining the points $(a,b)$ and $(c,d)$ has the same vertical displacement and opposite horizontal displacement as the line segment joining the points $(r,s)$ and $(u,v)$.
\end{proposition}

Costas arrays are difficult to construct and are conjectured to exist for all $n$; one may consult \cite{Drakakis06} for up-to-date information on their existence.  The smallest $n$ for which the existence of a Costas array is not known is $n=32$. It thus seems natural to investigate permutations under less restrictive conditions. Accordingly,
we define a permutation to be a {\it $k$-Costas permutation} provided that for each $i=0,1,\ldots,k$, its $i$th order differences are distinct. Thus an
$(n-1)$-Costas permutation is a Costas permutation. A 1-Costas permutation is a permutation whose discrete derivative consists of distinct integers. For more on Costas arrays, see \cite{Jedwab, Jedwab2, JedwabLin, SCH}.

In this paper we investigate various properties of the discrete derivative of a permutation, some of which are motivated by the classical derivative of a function. 
We now summarize the contents of this paper. In Section \ref{sec:dd}, we develop some basic properties of the discrete derivative of a permutation. In Section \ref{sec:local-global} we define the local and global variations of a permutation and determine their extreme values with characterizations of equality. Some other extremal questions, and the notion of convexity,  are treated in Section \ref{sec:other}.  Finally, in Section \ref{sec:coda} we  discuss some possible generalizations of the content of this paper.

\section{The discrete derivative}
\label{sec:dd}

The following observation is the analogue of the fact that a function is determined up to a constant by its derivative. Here,  since we are dealing with permutations in $\mc{S}_n$, no constant is involved.

\begin{proposition}
 \label{prop:uniqueness}
  A permutation $\pi \in\mc{S}_n$ is uniquely determined by its discrete derivative, i.e., the function $\pi: \mc{S}_n \rightarrow \mb{Z}^{n-1}$ given by $\pi \rightarrow D(\pi)$ is injective. 
\end{proposition}
\begin{proof}
 First, it is clear that $\pi$ is uniquely determined by the extension
 $(\pi_1, \pi_2-\pi_1,  \ldots, \pi_n-\pi_{n-1})$ of the discrete derivative,
 due to the expression
 \begin{equation}
  \label{eq:pi-sum}
    \pi_k=\pi_1+(\pi_2-\pi_1)+ \cdots + (\pi_k-\pi_{k-1}) \;\;\; (k \le n).
 \end{equation}   
 Next, when $\pi \in\mc{S}_n$ is given, then any other $\pi' \in \mc{S}_n$ with $D(\pi)=D(\pi')$, must be obtained by a shifting of $\pi$ in the sense that $\pi_i'=a+\pi_i \;\;(i \le n)$
 for some integer $a$.  But this implies that $a=0$; otherwise some $\pi'_i$ would not lie in the set $\{1, 2, \ldots, n\}$. Thus $\pi'=\pi$. 
\end{proof}

\begin{example}
\label{ex:first}
{\rm 
If
$\pi=(5,2,7,4,1,6,3) \in \mc{S}_7$, then its corresponding permutation matrix is 
\[
\left[
\begin{array}{c|c|c|c|c|c|c}
&&&&1&&\\ \hline
&1&&&&&\\ \hline
&&&&&&1\\ \hline
&&&1&&&\\ \hline
1&&&&&&\\ \hline
&&&&&1&\\ \hline
&&1&&&&
\end{array}
\right].
\]
In terms of integral points in the plane as previously described, there are only two different line segments 
of  the  six  determined by successive points. The discrete derivative of $\pi$ is $D(\pi)=(-3,5,-3,-3,5,-3)$. 
} \endproof
\end{example}

By Proposition \ref{prop:uniqueness} a permutation $\pi=(\pi_1,\pi_2,\ldots,\pi_n)\in\mc{S}_n$ is determined by the $(n-1)$ entries in row 1 of its difference triangle
$\Theta_{\pi}$. The permutation $\pi$ is, of course,  also determined by any $(n-1)$ of the entries of $\pi$ itself. There are other sets of $(n-1)$ entries of the difference triangle which determine its corresponding permutation.  
Consider the complete graph $K_n$ with vertices labeled $1,2,\ldots,n$. To each edge $\{i,j\}$  of $K_n$ with $i<j$  we give  the weight $(\pi_j-\pi_i)$, thereby obtaining a weighted complete graph $K_n^{\pi}$.


\begin{proposition}\label{prop:tree}
Let $\pi \in\mc{S}_n$. Then any weighted spanning tree $T^{\pi}$ of $K_{n}^{\pi}$ uniquely determines $\pi$.
\end{proposition}

\begin{proof} The weighted spanning tree $T^{\pi}$ has  $(n-1)$ edges. Let $p,q,r$ be  distinct vertices of $T^{\pi}$ with edges $\{p,q\}$ and $\{q,r\}$ where $p<q$ and $q<r$.  Then the weight of $\{p,q\}$ is $\pi_q-\pi_p$ and the weight of $\{q,r\}$ is $\pi_r-\pi_q$. Since $(\pi_r-\pi_q)+(\pi_q-\pi_p)=\pi_r-\pi_p$, the weight of the edge $\{p,r\}$ is also determined. Proceeding inductively like this, we see that the weights of all the edges of  $K_n^{\pi}$ are determined, in particular the weights of the edges $\{1,2\}, \{2,3\},\ldots,\{n-1,n\}$; thus  the first-order differences of $\pi$, and thus $\pi$ itself (see Proposition \ref{prop:uniqueness}),  are determined. 
\end{proof}

\medskip	
\begin{example}
\label{ex:costas2rep}
{\rm  
Consider  again $n=6$ and $\pi=(3,6,1,5,2,4)$. The difference triangle is 
\[
\begin{array}{rrrrrrrrrrr}
     3&&6&&1&&5&&2&&4\\ 
      &\cellcolor{blue!25}3&&\cellcolor{blue!25}-5&&4&&-3&&2&\\ 
       &&-2&&-1&&1&&\cellcolor{blue!25}-1&&\\ 
        &&&\cellcolor{blue!25}2&&\cellcolor{blue!25}-4&&3&&&\\ 
         &&&&-1&&-2&&&&\\ 
          &&&&&1&&&&& \end{array} .\]
         Consider the  spanning weighted tree of  $T^{\pi}$ of $K_{7}^{\pi}$ with (unweighted) edges
 $F=\{\{1,2\}, \{2,3\},\{4,6\},\{1,4\},\{2,5\} \}$ whose weights are highlighted. Then $F$ is a path $3,2,1,4,6$ with an additional edge $\{2,5\}$. A simple verification checks that the difference triangle is determined  by the highlighted entries.     \endproof  
}\end{example}

It is  natural to ask which vectors in $\mb{Z}^n$ are discrete derivatives of permutations in $\mc{S}_n$.

Define, for an integral vector $z=(z_1, z_2, \ldots, z_{n-1})$, the set 
 \begin{equation}\label{eq:sums2}
   S_z:=\{0\} \cup \left\{\sum_{k=1}^i z_k: 1 \le i \le  n-1\right\}.
 \end{equation}

\begin{proposition}
 \label{pr:char}
  Let  $z=(z_1, z_2, \ldots, z_{n-1})$ be an integral vector. Then $z$ is the discrete derivative of some permutation in $\mc{S}_n$ if and only if $S_z$ is a set of $n$ consecutive integers containing $0$. 
\end{proposition}
\begin{proof}
 Assume $z=D(\pi)$ for some $\pi \in \mc{S}_n$. Then, by (\ref{eq:pi-sum}), 
 \[
   \sum_{k=1}^i z_k =(\pi_2-\pi_1) + (\pi_3-\pi_2) + \cdots +(\pi_{i+1}-\pi_i)=\pi_{i+1}-\pi_1
 \]
 for $i \le n-1$. 
 So the numbers $\sum_{k=1}^i z_k$ $(i \le n-1)$ are 
 \[
    \pi_2-\pi_1, \pi_3-\pi_1, \ldots, \pi_n-\pi_1.
 \]
 Therefore $S_z$ consists of the numbers $\pi_1, \pi_2, \ldots, \pi_n$ with $\pi_1$ subtracted from each, and this is a set of $n$ consecutive integers containing $0$. 
 
 Conversely, let $S$ be a set of $n$ consecutive integers containing $0$, so 
  \[
    S=\{-s, -s+1, \ldots, -1, 0, 1, \ldots, n-s-1\}
 \]
 for some integer $s$ with $0 \le s \le n-1$. Let $\pi$ be the permutation
 \[
      \pi=(s+1, 1, 2, 3, \ldots, s, s+2, \ldots, n).
 \]
 Then $z:=D(\pi)=(-s, 1, 1, \ldots, 1, 2, 1, 1, \ldots, 1)$, where the entry $2$ occurs in position $s+1$. So the numbers $\sum_{k=1}^i z_k$ ($1 \le i \le  n-1$) become 
 \[
    -s, -s+1, -s+2, \ldots, -1, 1, 2, \ldots, n-s-2
 \]
 and therefore $S_z=\{-s, -s+1, \ldots, -1, 0, 1,  \ldots, n-s-1\}=S$, as desired. 
\end{proof}

\medskip

For a  permutation $\pi$ in ${\mathcal S}_n$, we call the set $S_{D(\pi)}$ the {\it sum-characteristic} of $\pi$.
There are $n$ sets of $n$ consecutive integers containing 0, namely,
\begin{equation}\label{eq:sumchar}
\{-(n-1),\ldots,-1,0\},\{ -(n-2),\ldots,0,1\}, \ldots, \{0,1,\ldots,n-1\},\end{equation}
and by Proposition \ref{pr:char}
each of the sets in (\ref{eq:sumchar}) is the sum-characteristic of a permutation
in ${\mathcal S}_n$. The sum-characteristic of a permutation $\pi=(\pi_1,\pi_2,\ldots,\pi_n)$ is uniquely determined by $\pi_1$ and thus is the sum-characteristic  of $(n-1)!$ permutations in ${\mathcal S}_n$.
To see this, we observe that as in the proof of Proposition \ref{pr:char},  the entries of the sum-characteristic are $\{\pi_2-\pi_1,\pi_3-\pi_1,\ldots,\pi_n-\pi_1\}$. Thus if two permutations have the same $\pi_1$, then the entries $\{\pi_2,\pi_3,\ldots,\pi_n\}$ are $\{1,2,\ldots,n\}\setminus \{\pi_1\}$ 
and hence their sum-characteristics are identical. 

\medskip

Let $\pi\in{\mathcal S}_n$ with discrete derivative $D(\pi)=(a_1,a_2,\ldots, a_{n-1})$. It follows from (\ref{eq:sums})  that $\pi$ is a Costas permutation if and only if
the integers in each of the sequences
\[a_i+a_{i+1}+\cdots+a_{i+k}\quad (1\le i\le i+k\le n-1;k\ge 0)\]
are distinct.
Hence by Proposition \ref{pr:char},  a Costas permutation of order
$n$ exists if and only if there is a sequence $(b_1,b_2,\ldots,b_{n-1})$ of $n-1$ integers which, with 0, can be reordered to form a consecutive set of integers such that 
the integers in each of the sequences
\[b_i+b_{i+1}+\cdots+b_{i+k}\quad (1\le i\le i+k\le n-1; k\ge 0)\]
are distinct. Thus the existence of Costas permutations can be regarded as a problem within additive number theory.

\begin{example}
\label{ex:first-returned}
{\rm 
Consider the permutation in Example \ref{ex:first}, so  
$\pi=(5,2,7,4,1,6,3) \in \mc{S}_7$ and $D(\pi)=(-3,5,-3,-3,5,-3)$. Then, if $z=D(\pi)$, we get 
\[
  S_{z}=\{0, -3,2,-1,-4,1,-2\}=\{-4,-3,-2,-1,0,1,2\}.
\]  
We now construct another permutation $\pi'$ such that $S_{z'}=S_z$ where $z'=D(\pi')$. Let $z'=(-4,1,1,1,2,1)$. Then $S_{z'}=\{-4,-3,-2,-1,0,1,2\}=S_{z}$, and, for instance, the permutation $\pi'=(5,1,2,3,4,6,7)$ is such that 
$D(\pi')=(-4,1,1,1,2,1)$.\endproof
} 
\end{example}

We now consider some questions related to the signs and values of the discrete derivative of a permutation. 
It is easy to see that there is only one permutation $\pi \in \mc{S}_n$ having only positive discrete derivatives, namely the {\it identity permutation} $\pi=\iota_n$ corresponding to the $n\times n$ identity matrix $I_n$.
Similarly,  the {\it anti-identity permutation} $\zeta_n=(n,\ldots,2,1)$ 
corresponding to the 
 backward identity matrix $L_n=[l_{ij}]$, where $l_{ij}=1$ when $j=n-i+1$ and 0 otherwise ($i \le n$),
is the only permutation with only negative derivatives.  A permutation in $\mc{S}_n$ is a {\em Grassmannian permutation} provided it has only one descent. Grassmannian permutations are the only permutations $\pi$ with only one negative entry in their discrete derivative. 
 Henceforth we generally refer to a discrete derivative simply as a derivative.
 
We now consider permutations 
 with only two  distinct values in their discrete derivative. 
An example of a permutation  all of whose whose derivative values are $5$ or $-3$ is given in Example \ref{ex:first}. Another similar example is 
$(6,3,5,2,4,1)$ where $D(\pi)=(-3,2,-3,2,-3)$. 

\medskip
Let $p$ and $q$ be distinct integers. If for some $n\ge 1$ there exists a permutation $\pi \in \mc{S}_n$,  such that $\{D(\pi)_i: i \le n-1\}=\{p,q\}$, then we say that $(p,q)$ is a {\em $D$-pair}. Let $(p,q)$ be a $D$-pair. Then $(q,p)$ is also a $D$-pair, and therefore we assume hereafter than $|p|\ge |q|$. In fact, we may also assume $p>0$ because, the reverse of a permutation has the same derivative values as the original but with opposite signs.

\begin{lemma}
  \label{lem:basic-D-pair}
    Assume that $n\ge 3$ and that $(p,q)$ is a $D$-pair. Then $p$ and $q$ have opposite signs, and $p$ and $q$ are relatively prime.
\end{lemma}
\begin{proof}
The only permutation with all derivatives positive (and thus all equal to $1$), is the identity permutation, and the only permutation with all derivatives negative (and thus all equal to $-1$), is the anti-identity permutation. So $p$ and $q$ have opposite signs. If $\gcd(p,q):=d \ge 2$, then each of $\pi_2,\pi_3,\ldots,\pi_n$ differ from $\pi_1$ by a multiple of $d$, an impossibility as $\pi$ is a permutation.
\end{proof}

We next show that the conditions on $p$ and $q$ discussed above provide a  characterization of   $D$-pairs. 

\begin{theorem}
 \label{thm:char-D2}
  Let $a$ and $b$ be relatively prime integers with $1 \le a<b\le n-1$. Then there exists an  integer $n$ such that $(a,-b)$ is a $D$-pair corresponding to a permutation in ${\mathcal S}_n$.
\end{theorem}
\begin{proof}
 We first treat the case when $a=1$. Let $\pi=(2, 3, \ldots, b+1,1) \in \mc{S}_{b+1}$. Then $D(\pi)=(1, 1, \ldots, 1, -b)$, so $(1,-b)$ is a $D$-pair corresponding to the permutation $\pi$. 
  
 Next, let $a \ge 2$ and
 let $n=a+b$. Define $p_i=1+(i-1)a$ for $i=1,2, \ldots, n$. Also let $\pi_i$ be (uniquely) defined by
 \begin{equation}
  \label{eq:pi-p}
      \pi_i \equiv p_i \mod(a+b), \;\, \mbox{\rm and} \;\; \pi_i \in \{1, 2, \ldots, n\}.
 \end{equation}
 We show that $\pi=(\pi_1, \pi_2, \ldots, \pi_n)$ is a permutation with the desired properties.
 
 Assume that $\pi_i = \pi_j$ for some $i,j$ with $1\le i<j\le n$. This implies that 
 \[
     1+(i-1)a  \equiv 1+(j-1)a \mod(a+b)
 \]
 and therefore 
 \[
     ia  \equiv ja \mod(a+b).
 \]
 Since $\gcd(a,a+b)=\gcd(a,b)=1$, we see that $i  \equiv j \mod(a+b)$ and hence  $i=j$. 
  This proves that $\pi$ is a permutation. 
  
  Next we consider the derivatives of $\pi$.   
 For each $i$ we consider two possibilities:
 
\noindent  {\em Case $1$:  $(s-1)(a+b) < p_i < p_{i+1}\le s(a+b)$ for some $s$}.  Then $D(\pi)_i=\pi_{i+1}-\pi_i=a$. 
 
\noindent  {\em Case $2$:  $p_i \le s(a+b)$ and $p_{i+1}> s(a+b)$ for some $s$}.  The  facts that  $p_{i+1}=p_{i}+a$ and $\pi$ is obtained by reducing the $p_i$'s modulo $a+b$ implies  that
 \[   \pi_{i+1}-\pi_i=a-(a+b)=-b.
\]

Note that Case 2 will occur for some $i$ as $p_n>n$. In fact, 
\[
   p_n=1+(a+b-1)a>a+b=n, 
\]
as $(a+b-1)a>a+b-1$ because $a \ge 2$. 

This shows that $\{D(\pi)_i: i \le n-1\}=\{a,-b\}$, and we conclude that  $(a,-b)$ is a $D$-pair and $\pi$ is a realization of $(a,-b)$.
\end{proof}

In terms of permutation matrices, 
the proof just given  constructs an $n\times n$  permutation matrix $P$ with $n=a+b$ that realizes a $D$-pair $(a,-b)$.  The construction is easy to describe: for $a\ge 2$ put a 1 in position $(1,1)$ and, row by row, move $a$ columns to the right computing column indices modulo $(a+b)$, that is, move cyclically from row to row. For $a=1$ we do the same, but start in position $(1,2)$. 
The following example illustrates this construction.

\begin{example}
\label{ex:two3}
{\rm 
Consider $a=5$, $b=13$.  Here 5 and 13 are relatively prime, so $(5,-13)$ is a $D$-pair. 
Then $n=a+b=18$. The construction just discussed then gives  the permutation 
\[
   \pi=(1,6,11,16,3,8,13,18,5,10,15,2,7,12,17,4,9,14)\in \mc{S}_{18}
\]
whose corresponding permutation matrix is 
\begin{equation}\label{eq:exD}
\left[
\begin{array}{c|c|c|c|c|c|c|c|c|c|c|c|c|c|c|c|c|c}
1&&&&&&&&&&&&&&&&&\\ \hline
&&&&&1&&&&&&&&&&&&\\ \hline
&&&&&&&&&&1&&&&&&&\\ \hline
&&&&&&&&&&&&&&&1&&\\ \hline \hline
&&1&&&&&&&&&&&&&&&\\ \hline
&&&&&&&1&&&&&&&&&&\\ \hline
&&&&&&&&&&&&1&&&&&\\ \hline
&&&&&&&&&&&&&&&&&1\\ \hline \hline
&&&&1&&&&&&&&&&&&&\\ \hline
&&&&&&&&&1&&&&&&&&\\ \hline
&&&&&&&&&&&&&&1&&&\\ \hline \hline
&1&&&&&&&&&&&&&&&&\\ \hline
&&&&&&1&&&&&&&&&&&\\ \hline
&&&&&&&&&&&1&&&&&& \\ \hline
&&&&&&&&&&&&&&&&1&\\ \hline \hline
&&&1&&&&&&&&&&&&&&\\ \hline
&&&&&&&&1&&&&&&&&&\\ \hline
&&&&&&&&&&&&&1&&&&\\ 
\end{array}
\right].
\end{equation}
The double horizontal lines indicate where the derivative changes from  positive to negative. So   \[D(\pi)=(5,5,5,-13,5,5,5,-13,5,5,-13,5,5,5,-13,5,5).\]
 Our construction is a simple generalization of the standard full-cycle permutation matrix ($a$ equals 1) on which the definition of a circulant matrix rests.
 
 The inverse of $\pi$ is given by
 \[\pi^{-1}=(1,12,5,16,9,2,13,6,17,10,3,14,7,18,11,4,15,8),\]
 and 
 \[D(\pi^{-1})=(11,-7,11,-7,-7,11,-7,11,-7,-7,11,-7,11,-7,-7,11,-7).\]
 Thus $\pi^{-1}$ realizes the $D$-pair $(11,-7)$. Since the permutation matrix corresponding to $\pi^{-1}$ is the transpose of the permutation matrix (\ref{eq:exD}), these differences result by cyclically considering the columns of (\ref{eq:exD}).
} \endproof
\end{example}

\begin{example}
{\rm 
Let $a=4$, $b=5$. Then $n=a+b=9$ and the construction above gives the following $D$-realization of the $D$-pair $(4,-5)$:
\[
\left[
\begin{array}{c|c|c|c|c|c|c|c|c}
1&&&&&&&&\\ \hline
&&&&1&&&&\\ \hline
&&&&&&&&1\\ \hline
&&&1&&&&&\\ \hline
&&&&&&&1&\\ \hline
&&1&&&&&&\\ \hline
&&&&&&1&&\\ \hline
&1&&&&&&&\\ \hline
&&&&&1&&&\\ 
\end{array}
\right].
\]
The permutation is $\pi=(1,5,9,4,8,3,7,2,6) \in \mc{S}_9$ and its discrete derivative is $D(\pi)=(4,4,-5,4,-5,4,-5)$.  We have that $\pi^{-1}=(1,8,6,4,2,9,7,5,3)$ and $D(\pi^{-1})=(7,-2,-2,-2,7,-2,-2,-2)$.
} \endproof
\end{example}

\begin{corollary}
 \label{cor:char-D2}
  Let $a$ and $b$ be relatively prime integers with $0<a<b$, and let $\pi$ be the permutation constructed in the proof of 
Theorem
 $\ref{thm:char-D2}$.
Let $a'$ be the inverse of $a$ modulo $a+b$, and let $b'=a+b-a'$. Then $(a',b')$ is  a $D$-pair corresponding to the permutation $\pi^{-1}$.
\end{corollary}

\section{Local and global variation of a permutation}
\label{sec:local-global}

In this section we define the  local and global variation of permutations and investigate some of their properties.

Let $\pi=(\pi_1,\pi_2,\ldots,\pi_n)$ be a permutation of $\{1,2,\ldots,n\}$. Then
\begin{itemize}
\item[\rm (a)] The {\it local variation} of $\pi$ is given by 
\[\delta(\pi)=\max\{|\pi_{i+1}-\pi_i|:1\le i\le n-1\},\]
the maximum absolute value of the  derivative values of $\pi$.
\item[\rm (b)] The {\it global variation} of $\pi$ is given by
\[\Delta(\pi)=\sum_{i=1}^{n-1} |\pi_{i+1}-\pi_i|=\|D(\pi)\|_1,\]
the $\ell_1$-norm of the discrete derivative.
\end{itemize}

The following proposition is clear.

\begin{proposition}
 \label{pr:min-max-local}
 For a permutation $\pi\in \mc{S}_n$,
 \[
    1\le \delta(\pi)\le  n-1
 \]
 with equality on the lower end if and only if $\pi=\iota_n\mbox{ or }\zeta_n$ and equality 
 on the upper end if  and only if $1$ and $n$ are adjacent in $\pi$.\hfill{$\Box$}
\end{proposition}

Now consider a   1-Costas permutation $\pi\in \mc{S}_n$ whose derivative values are thus distinct.  Then the lower bound is not attainable if $n\ge 3$.  Since there are $(n-1)$ differences and they can be positive or negative,  then 
\begin{equation}
 \label{eq:delta-lb}
   \delta(\pi) \ge \lceil n/2\rceil.
\end{equation}
In fact, if $n$ is even, the smallest $(n-1)$ values that the derivative could have are  $\pm 1, \pm 2,\ldots,\pm (\frac{n}{2}-1)$ and either $\frac{n}{2}$ or $-\frac{n}{2}$. Thus (\ref{eq:delta-lb}) holds when $n$ is even.
Next,  assume that $n$ is odd. If $\delta(\pi)=\frac{n-1}{2}$, then the values of the derivative $D(\pi)$ are $\{\pm 1, \pm 2,\ldots,\pm \frac{n-1}{2}\}$ and hence $\sum_{i=2}^n (\pi_{i}-\pi_{i-1})=0$.  Thus $\sum_{i=2}^n \pi_{i}=\sum_{i=2}^n\pi_{i-1}$ implying that $\pi_1=\pi_n$, a contradiction. Thus if $n$ is odd, $\delta(\pi)\ge \frac{n+1}{2}=\lceil n/2 \rceil$, and (\ref{eq:delta-lb}) is verified.
For example, with $n=5$, $\pi=(1,3,4,2,5)$ is a 1-Costas permutation with $D(\pi)=(2,1,-2,3)$ and hence $\delta(\pi)=3$.

We show how to construct $1$-Costas permutations attaining the lower bound (\ref{eq:delta-lb}) on the local variation. 
Let $k\ge 1$. If $k$ is even, say $k=2p$, we define  $\pi_{(k)} \in \mc{S}_k$  by
\[
   \pi_{(k)}=(p, p+1, p-1,p+2, p-2, p+3, p-3, \ldots, 1,k).
\]
If $k$ is odd, say $k=2p+1$, we define  $\pi_{(k)} \in \mc{S}_k$  by
\[
   \pi_{(k)}=(p+1, p+2, p,p+3, p-1, p+4, p-2, \ldots, k,1).
\]
Thus, the derivative of $\pi_{(k)}$ is 
\[
   D(\pi_{(k)})=(1,-2,3,-4, 5, -6, \ldots, \pm (k-1))
\]
where the signs alternate. 
Let
%
$\Pi_k=P_{\pi_{(k)}}$
%
be the permutation matrix corresponding to $\pi_{(k)}$. 
For instance, $\pi_{(4)}=(2,3,1,4)$ and  $\pi_{(5)}=(3,4,2,5,1)$, and 
\[
\Pi_4=
\left[
\begin{array}{c|c|c|c}
&1&&\\ \hline
&&1&\\ \hline
1&&&\\ \hline
&&&1
\end{array}
\right], \;\;
\Pi_5=
\left[
\begin{array}{c|c|c|c|c}
&&1&&\\ \hline
&&&1&\\ \hline
&1&&&\\ \hline
&&&&1\\ \hline
1&&&&
\end{array}
\right].
\]

Recall that $L_k$ is the backward identity matrix defined in the introduction. Observe that, for each permutation matrix $P \in \mc{S}_k$
\[
   D(L_kP)=-D(P).
\]
An example of a 1-Costas permutation attaining the lower bound above is 
$\pi=(3, 2, 4, 1, 5, 10, 6, 9, 7, 8)$ with derivative $D(\pi)=(-1, 2, -3, 4, 5, -4, 3, -2, 1)$. 
More generally we have the following proposition. In connection with the proof, it can be useful to consider the examples below.

\begin{theorem}
\label{thm:costas}
Let $n\ge 2$ be a positive integer. Then there exists a $1$-Costas permutation $\pi^* \in \mc{S}_n$ with $\delta(\pi^*)=\lceil \frac{n}{2} \rceil$, the smallest possible value.  Such a minimizing permutation $\pi^*$ depends on the parity of $n$ and is given by:

$(i)$ $n$ is even, say $n=2k$$:$  $\pi^*$ corresponds to the permutation matrix 
\begin{equation}
 \label{eq:min-delta_1}
    P_{\pi^*}=\Pi_k \oplus L_{k} \Pi_k  ,
\end{equation}

$(ii)$ $n$ is odd, and $n=2k+1$ with $k$ even$:$  $\pi^*$ corresponds to the permutation matrix 
\begin{equation}
 \label{eq:min-delta_2}
    P_{\pi^*}=
    \left[
    \begin{array}{cc}
       O &     L_{k+1}\Pi_{k+1} \\*[\smallskipamount]
       \Pi_k  & O
    \end{array}    
    \right], 
\end{equation}

$(iii)$ $n$ is odd, and $n=2k+1$ with $k$ odd$:$  $\pi^*$ corresponds to the permutation matrix 
\begin{equation}
 \label{eq:min-delta_3}
    P_{\pi^*}=
    \left[
    \begin{array}{cc}
       O &     L_{k+1}\Pi_{k+1}L_{k+1} \\*[\smallskipamount]
       \Pi_kL_{k}  & O
    \end{array}    
    \right].
\end{equation}

Moreover, in each case, $\pi^*$ attains the minimum value of $\Delta(\pi)$ for $1$-Costas permutations, and this minimum value is $\frac{n^2}{4}$ when $n$ is even, and 
$\frac{(n-1)^2}{4}+1$ when $n$ is odd.
\end{theorem}
\begin{proof}
 (i) Assume that $n$ is even, say $n=2k$ for some positive integer $k$.  Consider the matrix $P_{\pi^*}$ in (\ref{eq:min-delta_1}). If $k$ is even, the derivative of  $P_{\pi^*}$ is 
 \[
     D(\pi^*)=(1, -2, 3, -4, \ldots, k-1, k, -(k-1), k-2, \ldots, 2,-1).
 \]
 If $k$ is odd, the derivative of  $P_{\pi^*}$ is 
 \[
     D(\pi^*)=(1, -2, 3, -4, \ldots, -(k-1), k, k-1, -(k-2), \ldots, 2,-1).
 \]
 Thus, in either case, $\pi^*$ is $1$-Costas, and $\delta(\pi^*)=k=\lceil n/2\rceil$.
 
 (ii) Next, assume $n$ is odd and $n=2k+1$ for some even positive integer $k$.  Consider the matrix $P_{\pi^*}$ in (\ref{eq:min-delta_2}).  Then 
 \[
   D(\pi^*)_{k+1}=\pi^*_{k+2}-\pi^*_{k+1}=
    k/2  - (k+ k/2+1)=-(k+1), 
 \]
  so the derivative of  $P_{\pi^*}$ is 
 \[
     D(\pi^*)=(k, -(k-1),k-2,  \ldots, 2, -1, -(k+1), 1,-2, 3, -4, \ldots, k-1).
 \]
 Thus, $\pi^*$ is $1$-Costas and $\delta(\pi^*)=k+1=\lceil n/2\rceil$.
 
 (iii) Assume  $n=2k+1$ for some odd $k$, and consider the matrix $P_{\pi^*}$ in (\ref{eq:min-delta_3}).  Then 
 \[
   D(\pi^*)_{k+1}=\pi^*_{k+2}-\pi^*_{k+1}=
    (k+1)/2  - (k+ (k+1)/2+1)=-(k+1), 
 \]
  so the derivative of  $P_{\pi^*}$ is 
 \[
     D(\pi^*)=(k, -(k-1),k-2,  \ldots, 2, -1, -(k+1), 1,-2, 3, -4, \ldots, k-1).
 \]
 Thus,  $\pi^*$ is $1$-Costas and $\delta(\pi^*)=k+1=\lceil n/2\rceil$.

 Finally, we consider $\Delta(\pi^*)$. When $n=2k$ is even, the derivatives are
 \[
     \{\pm 1, \pm 2, \ldots, \pm k\}
 \]
  so clearly the minimum value of $\Delta(\pi)$ is then attained among $1$-Costas permutations. 
  When $n=2k+1$ is odd, the derivatives are
 \[
     \{\pm 1, \pm 2, \ldots, \pm (k-1), k, -(k+1)\}
 \]
  which is the minimum value of $\Delta(\pi)$  among $1$-Costas permutations as some derivative is at least $k+1$ in absolute value due to (\ref{eq:delta-lb}). 
 \end{proof}

We now give three examples illustrating the three cases in Theorem \ref{thm:costas}.

\begin{example}
 {\rm 
 Let $n=12$, so $\lceil n/2\rceil=6$. Then  the matrix $P_{\pi^*}$ in (\ref{eq:min-delta_1}) is 
\[\left[\begin{array}{c|c|c|c|c|c||c|c|c|c|c|c}
&&1&&&&&&&&&\\ \hline
&&&1&&&&&&&&\\ \hline
&1&&&&&&&&&&\\ \hline
&&&&1&&&&&&&\\ \hline
1&&&&&&&&&&&\\ \hline
&&&&&1&&&&&&\\ \hline\hline
&&&&&&&&&&&1\\ \hline
&&&&&&1&&&&&\\ \hline
&&&&&&&&&&1&\\ \hline
&&&&&&&1&&&&\\ \hline
&&&&&&&&&1&&\\ \hline
&&&&&&&&1&&&\end{array}\right]=\left[\begin{array}{cc} A_1&O\\O&A_2\end{array}\right],\]
where the rows of $A_2$ are in the reverse order of those of $A_1$.
The derivative is computed as $D(\pi^*)=(1,-2,3,-4,5,6,-5,4,-3,2,-1)$, and  $\delta(\pi^*)=6$.
} \endproof
\end{example}

\begin{example}
 {\rm 
 Let $n=13$, so $\lceil n/2\rceil=7$. Then  the matrix $P_{\pi^*}$ in (\ref{eq:min-delta_2}) is 
\[\left[\begin{array}{c|c|c|c|c|c||c|c|c|c|c|c|c}
&&&&&&1&&&&&&\\ \hline
&&&&&&&&&&&&1\\ \hline
&&&&&&&1&&&&&\\ \hline
&&&&&&&&&&&1&\\ \hline
&&&&&&&&1&&&&\\ \hline
&&&&&&&&&&1&&\\ \hline
&&&&&&&&&1&&&\\ \hline \hline
&&1&&&&&&&&&&\\ \hline
&&&1&&&&&&&&&\\ \hline
&1&&&&&&&&&&&\\ \hline
&&&&1&&&&&&&&\\ \hline
1&&&&&&&&&&&&\\ \hline
&&&&&1&&&&&&&\\ 
\end{array}\right]=\left[\begin{array}{cc} O & A_1\\A_2 &O\end{array}\right].\]
The derivative is computed as $D(\pi^*)=(6,-5,4,-3,3,-1,-7, 1, -2,3,-4,5)$, and $\delta(\pi^*)=7$.
} \endproof
\end{example}

\begin{example}
 {\rm 
 Let $n=11$, so $\lceil n/2\rceil=6$. Then  the matrix $P_{\pi^*}$ in (\ref{eq:min-delta_3}) is 
\[\left[\begin{array}{c|c|c|c|c||c|c|c|c|c|c}
&&&&&1&&&&&\\ \hline
&&&&&&&&&&1\\ \hline
&&&&&&1&&&&\\ \hline
&&&&&&&&&1&\\ \hline
&&&&&&&1&&&\\ \hline
&&&&&&&&1&&\\ \hline \hline
&&1&&&&&&&&\\ \hline
&1&&&&&&&&&\\ \hline
&&&1&&&&&&&\\ \hline
1&&&&&&&&&&\\ \hline
&&&&1&&&&&
\end{array}\right]=\left[\begin{array}{cc} O & A_1\\A_2 &O\end{array}\right].\]
The derivative is computed as $D(\pi^*)=(5,-4,3,-2,1,-6,-1,2,-3,4)$, and  $\delta(\pi^*)=6$.
} \endproof
\end{example}

We now determine the extreme values of the {\em global variation} for general permutations.
Clearly, $\min_{\pi \in \mc{S}_n} \Delta(\pi)=n-1$ and this minimum is attained only for the identity and the anti-identity permutations. The problem of maximizing $\Delta(\pi)$ is more complex.  
Define
\[
    \Delta^*_n=\max\{\Delta(\pi): \pi \in \mc{S}_n\}.
\]
It is convenient to treat the even and odd cases separately.

Let $n$ be even, say $n=2k$, and let $\pi=(\pi_1, \pi_2, \ldots, \pi_n)$ be a permutation. We say that $\pi$ is {\em mid-alternating} if for all $i<n$ the consecutive entries of $\pi$ satisfy either (i) $\pi_i \le k$, $\pi_{i+1} \ge k+1$, or (ii) $\pi_i \ge k+1$, $\pi_{i+1} \le k$.

\begin{example}
\label{ex:Delta-max1}
{\rm Let $n=8$, so $k=4$. The permutation $\pi=(4,6,2,7,3,8,1,5)$ is mid-alternating, and the corresponding permutation matrix is 
\[
 P= 
 \left[
  \begin{array}{l|l|l|l||l|l|l|l}
   &&&1&&&& \\ \hline
   &&&&&1&& \\ \hline
   &1&&&&&& \\ \hline
   &&&&&&1& \\ \hline
   &&1&&&&& \\ \hline
   &&&&&&&1 \\ \hline
   1&&&&&&& \\ \hline
   &&&&1&&& 
 \end{array}
 \right].\]
 Of two 1's in consecutive rows, one is to the left of the double-vertical line and one is to the right.
 \endproof
}
\end{example}

\begin{theorem}
\label{thm:extreme-Delta-even}
 Let $n$ be even,  and let $\pi =(\pi_1, \pi_2, \ldots, \pi_n)\in \mc{S}_n$. Then $\Delta(\pi)=\Delta^*_n$ if and only if $\pi$ is mid-alternating and  
 $\{\pi_1, \pi_n\}=\{\frac{n}{2},\frac{n}{2}+1\}$. 
   Moreover, $\Delta_n^*=(n^2-2)/2$. 
 \end{theorem}
\begin{proof}
 Let $\pi=(\pi_1, \pi_2, \ldots, \pi_n)$ be a permutation. In the expression for $\Delta(\pi)$ we replace each $|\pi_{i+1}-\pi_i|$ by $(\pi_{i+1}-\pi_i)$ times the sign of this difference. This gives 
  %
 \begin{equation}
  \label{eq:decomp}
   \Delta(\pi)=\sum_{i=1}^{n-1} |\pi_{i+1}-\pi_{i}|= 
   2\sum_{i \in V_+} \pi_{i}+0\sum_{i \in V_0} \pi_{i} -2\sum_{i \in V_-} \pi_{i} +s_1\pi_1+s_n\pi_n
 \end{equation}
 where $V_+$ (resp. $V_-$; $V_0$) are those $i \in \{2, 3, \ldots, n-1\}$ such that $\pi_i$ is larger than both $\pi_{i-1} $ and $\pi_{i+1}$ (resp., smaller; in between), and $s_1, s_n = \pm 1$.  
 
 Note that $|V_+| \le n/2-1$ as $i$ and $i+1$ cannot both belong to $V_+$. Similarly, $|V_-| \le n/2-1$. 
  It therefore follows from (\ref{eq:decomp}) that an upper bound on  $\Delta(\pi)$  is the sum of twice the $(n/2-1)$  largest integers in $\{1,2,\ldots,n\}$ with  the next largest integer $(n/2+1)$, and subtracting twice the sum of the $(n/2-1)$ smallest integers in $\{1,2,\ldots,n\}$ with the next smallest integer $(n/2)$. In fact, here we replace  some of the zeros in (\ref{eq:decomp}) corresponding to $V_0$ by the difference of $k$th largest and the $k$th smallest number in $\{1, 2, \ldots, n\}$, where $k<n/2$, so this difference is positive. 
   Thus $\Delta(\pi)$ is bounded by $\alpha-\beta$ where
\begin{equation}
\alpha=
2\left( n+(n-1)+\cdots+\left(\frac{n}{2}+2\right)\right) +\left(\frac{n}{2}+1\right), 
\end{equation}
\begin{equation}
\beta=2\left( 1+2+\cdots+\left(\frac{n}{2}-1\right)\right)+\frac{n}{2}.
\end{equation}
Note, for clarity,  that because of the factor of 2 in both $\alpha$ and $\beta$, each is the sum of $(n-1)$ integers taken from $\{1,2,\ldots,n\}$
An elementary computation gives that 
\[
 \alpha-\beta=\frac{n^2-2}{2}, 
 \]
and, as $\pi$ was arbitrary,  we have shown that $\Delta_n^*\le (n^2-2)/2$. 

In order that $\Delta(\pi) =(n^2-2)/2$, for a specific permutation $\pi$, it  follows from our bounding argument that $V_0=\emptyset$, and therefore the signs of the derivatives must alternate (as consecutive entries cannot both be 2, or both be $-2$). 
Moreover, each of $\pi_i$ ($i \in V_+$)  must be more that $n/2$ and each of $\pi_i$ ($i \in V_-$) must be less than or equal to $n/2$. In addition, since it is $\pi_1$ and $\pi_n$ that enter only once in the computation of $\Delta(\pi)$, we must have $\{\pi_1,\pi_n\}=\{n/2,n/2+1\}$ where, if $\pi_1=n/2$, then $\pi_2>n/2$, then $\pi_3<n/2$ and so forth, while if $\pi_1=n/2+1$, then $\pi_2<n/2+1$, then $\pi_3>n/2+1$, and so forth, that is, $\pi$ must be mid-alternating. 
\end{proof}

The permutation $\pi$ in Example \ref{ex:Delta-max1} satisfies the conditions of Theorem \ref{thm:extreme-Delta-even}, so $\Delta(\pi)=\Delta^*_8=31$.

We turn to the case when $n$ is odd, say $n=2k+1$,  and let $\pi=(\pi_1, \pi_2, \ldots, \pi_n)$ be a permutation. We say that $\pi$ is {\em mid-alternating} if for all $i<n$ the consecutive entries of $\pi$ satisfy either (i) $\pi_i \le k+1$, $\pi_{i+1} \ge k+1$, or (ii) $\pi_i \ge k+1$, $\pi_{i+1} \le k+1$.

The following result may be shown using the same type of arguments as in the proof of Theorem \ref{thm:extreme-Delta-even}, so we therefore omit the proof. 

\begin{theorem}
\label{thm:extreme-Delta-odd}
 Let $n$ be odd, 
 and let $\pi =(\pi_1, \pi_2, \ldots, \pi_n)\in \mc{S}_n$. Then $\Delta(\pi)=\Delta^*_n$ if and only if $\pi$ is mid-alternating and  
 $\{\pi_1, \pi_n\}$ equals either $\{\frac{n-1}{2},\frac{n+1}{2}\}$ or $\{\frac{n+1}{2}, \frac{n+3}{2}\}$.
   Moreover, $\Delta_n^*=\frac{3n^2-6n-13}{4}$.
 \end{theorem}

\begin{example}
{\rm The permutation $\pi=(4,5,2,7,1,6,3)$ satisfies the conditions of Theorem \ref{thm:extreme-Delta-odd}, and $\Delta(\pi)=\Delta^*_7=23$. The corresponding permutation matrix is 
\[
 P= 
 \left[
  \begin{array}{l|l|l||l||l|l|l}
   &&&1&&& \\ \hline
   &&&&1&& \\ \hline
   &1&&&&& \\ \hline
   &&&&&&1 \\ \hline
   1&&&&&& \\ \hline
   &&&&&1& \\ \hline
   &&1&&&& \\ 
 \end{array}
 \right].
\] \endproof
}
\end{example}


\section{Other properties of the discrete derivative}
\label{sec:other}

  We observe  that 
  $\min_{\pi \in \mc{S}_n} \,\max_i |D(\pi)_i| =1$,  and this minimum is attained for the identity and anti-identity permutations. Now we determine $\max_{\pi \in \mc{S}_n} \,\min_i |D(\pi)_i|$.

\begin{theorem}
 \label{thm:minmax1}
 \begin{equation}
 \label{eq:minmax1}
    \max_{\pi \in \mc{S}_n} \,\min_i |D(\pi)_i|  =\lfloor n/2 \rfloor.
 \end{equation}
\end{theorem}
\begin{proof}
Define $\hat{D}^{(n)}:=     \max_{\pi \in \mc{S}_n} \,\min_i |D(\pi)_i|$. 
We initially prove that $\hat{D}^{(n)} \le \lfloor n/2 \rfloor$.

Let $\pi \in \mc{S}_n$, and define $k=\lfloor n/2 \rfloor$.  Assume $\min_i |D(\pi)_i| \ge k+1$. Let $P$ be the permutation matrix corresponding to $\pi$.  Let $i$ be the row of the unique 1 in column $k+1$. Then $i$ has at least one adjacent row, say it  is row $i-1$  (the argument is similar if it is row $i+1$, or both). Row $i-1$ has a unique 1, say in column $p$. But then
\[
   |D(\pi)_{i-1}|=|\pi_{i}-\pi_{i-1}| =|k+1-p|\le k,
\]
a contradiction.  

Therefore $\min_i |D(\pi)_i| \le k=\lfloor n/2 \rfloor$ and 
\[
\hat{D}^{(n)}=\max_{\pi \in \mc{S}_n} \,\min_i |D(\pi)_i| \le \lfloor n/2 \rfloor.
\]
It remains to construct a permutation $\pi \in \mc{S}_n$ with 
$\min_i |D(\pi)_i| = \lfloor n/2 \rfloor$. 

If $n$ is even, say $n=2k$, let 
\[ 
  \pi^{(n)}=(k+1, 1, k+2, 2, k+3, 3, \ldots, n, k).
\]
Then $D(\pi^{(n)})=(k,k+1, k, k+1, \ldots, k)$, so $\min_i |D(\pi)_i| = k$, as desired. 

If $n$ is odd, say $n=2k+1$, let 
\[ 
  \pi^{(n)}=(1, k+1, 2, k+2, 3, k+3, \ldots, k-1, n-1,k+1).
\]
Then $D(\pi^{(n)})=(k+1, k, k+1, k, \ldots, k,k+1)$, so $\min_i |D(\pi)_i| = k$, as desired. 
\end{proof}

Let $P^{(n)}$ be the permutation matrix corresponding to the extreme permutation $\pi^{(n)}$ constructed in the proof of Theorem \ref{thm:minmax1}. Note that when $n$ is even 
\[
   P^{(n+1)}=J_1 \oplus P^{(n)}.
\]
\begin{example} 
{\rm 
The extreme permutation matrices $P^{(6)}$ and $P^{(7)}$
are given by
 \[
 P^{(6)}=
\left[
\begin{array}{c|c|c|c|c|c}
&&&1&&\\ \hline
1&&&&&\\ \hline
&&&&1&\\ \hline
&1&&&&\\ \hline
&&&&&1\\ \hline
&&1&&&\\ 
\end{array}
\right] \;\mbox{\rm and}\;
P_2=
\left[
\begin{array}{c|c|c|c|c|c|c}
1&&&&&&\\ \hline
&&&&1&&\\ \hline
&1&&&&&\\ \hline
&&&&&1&\\ \hline
&&1&&&&\\ \hline
&&&&&&1\\ \hline
&&&1&&&\\ 
\end{array}
\right].
\]
Here $\hat{D}^{(6)}=\hat{D}^{(7)}=3$. The permutation in Example \ref{ex:first} also attains $\hat{D}^{(7)}$. 
 }\endproof
 \end{example}

 Next we discuss an analogue of convexity for the discrete derivative.
 We say that a permutation $\pi=(\pi_1,\pi_2,\ldots,\pi_n) \in \mc{S}_n$ and its corresponding permutation matrix $P=P(\pi)$ are {\em convex} provided its derivatives are increasing, i.e., 
 \begin{equation}
  \label{eq:convex}
  \pi_2-\pi_1 \le \pi_3-\pi_2 \le \cdots \le \pi_n-\pi_{n-1}.
 \end{equation}
 This is equivalent to 
 \[
   \pi_i \le \frac{1}{2}(\pi_{i-1}+\pi_{i+1}) \;\;(2 \le i \le n-1). 
 \]

 For instance, both the identity matrix and the backward identity matrix are convex. 
 A class of convex permutation matrices are obtained by a modification of the matrix $\Pi_k$ defined before Proposition \ref{thm:costas}. Let $\Pi^*_k$ be obtained from $\Pi_k$ by a plane rotation of the matrix by a counter-clockwise rotation of $90$ degrees.  For example,
 \[\Pi^*_6=
 \left[\begin{array}{c|c|c|c|c|c}
 &&&&&1\\ \hline
 &&&1&&\\ \hline
 &1&&&&\\ \hline
 1&&&&&\\ \hline
 &&1&&&\\ \hline
 &&&&1&\end{array}\right]\]
 which corresponds to the permutation $(6,4,2,1,3,5)$ with derivative $(-2,-2,-1,2,2)$. Then we see that $\Pi^*_k$ is a convex permutation matrix for every $k$.
  
 Let $P=[p_{ij}]$ be a $n \times n$   subpermutation matrix, i.e., a $(0,1)$-matrix with at most one 1 in every row and column. If $P$ contains a total of $k$ 1's, then $P$ corresponds to a subsequence $(i_1,i_2,\ldots,i_k)$ of a permutation of  $\{1,2,\ldots,n\}$.   Define 
$I_k(P)=\{i: p_{ij}=1 \;\mbox{\rm for some $j \le k$} \}$, the set of rows containing a 1 in the first $k$ columns. An {\em interval} in a set $\{1, 2, \ldots, n\}$ is a set of consecutive integers $I=\{k, k+1, \ldots, l\}$ for some $1 \le k \le l \le n$, and its {\em length} is $|I|=l-k+1$.

\begin{lemma}
 \label{lem:convex-int}
If $P$ is a convex permutation matrix of order $n$, then $I_k(P)$ is an interval of length $k$ for each $k \le n$.
\end{lemma}
\begin{proof}
 Let $\pi$ be the permutation corresponding to $P$. For $k=1$ the statement is clearly true. So, assume $k \ge 2$ and that $I_k(P)$ is not an interval. Then there are $i_1, i_2, i_3 \in I_k(P)$ with $i_1 < i_2 < i_3$ and the submatrix $P_1$ consisting of the first $k$ columns of $P$ has a 1 in rows $i_1$ and $i_3$, but not in row $i_2$. This clearly implies that there must exist an $s$ such  $\pi_{s+1}-\pi_s >0 > \pi_{s+2}-\pi_{s+1}$. So, the  derivative is not increasing, and $P$ is not convex, a contradiction. Therefore, $I_k(P)$ is  an interval.
\end{proof}

Note that the converse of the implication in Lemma \ref{lem:convex-int} is not true; for instance, consider the permutation matrix
 \[
 P=
 \left[
 \begin{array}{c|c|c|c}
 &&&1\\ \hline
  &&1&\\ \hline
  1&&&\\ \hline
   &1&&
  \end{array}
  \right].
  \]
  Then $I_k(P)$ is an interval of length $k$ for each $k \le n$, but $P$ is not convex.

 Let $P$ be an $n \times n$ subpermutation matrix.
   Let $k\le n$. We say that $P=[p_{ij}]$ is {\em $k$-convex} if (i) each of the first $k$ columns contains exactly one 1,  (ii) $I_k(P)$ is an interval, say equal to $\{r, r+1, \ldots, s\}$, and (iii) $\pi_{i+1}-\pi_i \le \pi_{i+2}-\pi_{i+1}$ for $i=r, r+1, \ldots, s-2$ where $\pi_i$ is the unique column in $P$ containing a 1 in row $i$. Now, let $P$ be such a subpermutation matrix which is $k$-convex and where columns $k+1, \ldots, n$ are all zero. Define the following (possibly empty) set $I^*_k(P)$ of cardinality at most 2:

(i) Let $r-1 \in I^*_k(P)$ if  $r>1$, and the matrix obtained from $P$ by putting a 1 in position $(r-1,k+1)$ is $(k+1)$-convex (this means that the derivative in row $r-1$ is less that the derivative in row $r$); 
  
  (ii) Let $s+1 \in I^*_k(P)$ if  $s<n$, and the matrix obtained from $P$ by putting a 1 in position $(s+1,k+1)$ is $(k+1)$-convex (this means that the derivative in row $s-1$ is less that the derivative in row $s$). 
    
  It follows from Lemma \ref{lem:convex-int} that if $P$ is a convex permutation matrix, then $P$ is also $k$-convex for each $k \le n$. 
  We use this property to construct convex permutation matrices of order $n$ by the following algorithm.

\begin{tabbing}
{\bf Al}\={\bf gorithm 1.} \\

\> 1. \mbox{ } \=Initialize: Let $P=[p_{ij}]=O$. \\
 \>2. \>Choose $i \in \{1, 2, \ldots, n\}$, and let $p_{i1}=1$. \\
 \>3. \>for \=$k=1, 2, \ldots, n-1$ \\
 \>\>\>a. \=If $I^*_k(P)= \emptyset$, stop. \\
 \>\>\> b.\>Otherwise, choose $i \in I^*_k(P)$ and let $p_{i,k+1}=1$. \\
  \>\>end
\end{tabbing}

\medskip
\begin{lemma}
 \label{lem:convex}
 If Algorithm $1$ does not terminate prematurely, i.e., Step $3$a does not occur, then  the resulting permutation matrix is convex. 
Any convex permutation matrix may be constructed by Algorithm $1$. 
\end{lemma}
\begin{proof}
We may assume $n \ge 2$. Consider Algorithm 1, and let $p_{i1}=1$. For $k=1$, we get 
\[
  I^*_1(P)= \left\{
  \begin{array}{ll}
     \{2\}  & \mbox{\rm if } i=1, \\
    \{i-1,i+1\}  & \mbox{\rm if } 1<i<n, \\
    \{n-1\}  & \mbox{\rm if } i=n.  
  \end{array}
  \right.
\]
Thus, after the step for $k=1$, the resulting matrix $P$ is $2$-convex.
It is not hard to see that the conditions on the set $I^*_k(P)$ assure that, when this set is nonempty for each $k$, the final matrix $P$ constructed will be a permutation matrix with  increasing derivatives and therefore it is convex.

Next, let $Q=[q_{ij}]$ be a convex permutation matrix. We need to show that $Q$ may be constructed by Algorithm 1 by suitable choice of the element in Step 3b in each iteration. Assume that Algorithm 1, after $k$ iterations, has constructed a matrix $P$ whose first $k$ columns coincide with those of $Q$ (for $k=1$ this is clear). Thus, $I_k(P)$ and $I_k(Q)$ are equal, say equal to $\{r, r+1, r+k-1\}$. Moreover, as $I_{k+1}(Q)$ is also an interval, the unique 1 in column $k+1$ of $Q$ must be in row $r-1$ or $r+k$, so assume first it is in row $r-1$. Since $Q$ is convex, $\pi=\pi(Q)$ satisfies $\pi_r-\pi_{r-1} \le \pi_{r+1}-\pi_{r}$ which means that $r-1 \in I^*_k(Q)=I^*_k(P)$, and therefore, in Algorithm 1, we can let column $k+1$ of $P$ have its 1 in  row $r-1$.  A similar construction works when 
the 1 in column $k+1$ of $Q$ is in  row  $r+k$. Thus, in any case, the $k+1$ first columns of $P$ equal the corresponding columns of $Q$. So, by induction, we may obtain $P=Q$ by suitable such choices in  Algorithm 1.
\end{proof}

\begin{theorem}
 \label{thm:convex}
  The set of convex permutations of order $n$ consists of 
  
 $(i)$ $(1, 2, \ldots, n)$ $($identity$)$,
 
 $(ii)$ $(n,1,2, \ldots, n-1)$,
 
 $(iii)$ $(n-1,1,2, \ldots, n-2,n)$,
 
  $(iv)$ $\Pi^*_n$,
 
 \noindent and the permutations obtained by reversing the order in each of these permutations.
 \end{theorem}
\begin{proof}
 We consider Algorithm 1, and construct a convex matrix $P=[p_{ij}]$ and corresponding permutation $\pi=\pi(P)$.  By symmetry, we may assume 
 \begin{equation}
  \label{eq:k-eq}
     p_{i1}=p_{i+1,2}=\cdots = p_{i+k-1,k}=1
 \end{equation}
 for some $i$, and with $k\ge 2$ maximal with this property. We discuss different cases.
 
 {\em Case} $1$: $k=n$.  Then $i=1$ and $P=I_n$. 
 
 {\em Case} $2$: $k=n-1$.  Then  $i=2$ and $\pi(P)=(n,1,2, \ldots, n-1)$.
 
 {\em Case} $3$: $k=n-2$. Then it is easy to see that, by convexity, that the only possibility is $i=2$. This gives the permutation $(n-1,1,2, \ldots, n-2,n)$.
 
 {\em Case} $4$: $3 \le k\le n-3$. Then (\ref{eq:k-eq}) holds and $p_{i-1,k+1}=1$. Then $p_{i-2, k+2} =0$, otherwise  $p_{i-2, k+2} =1$ and then $\pi_{i-1}-\pi_{i-2}=k+1-(k+2)=- 1$ and 
  $\pi_{i}-\pi_{i-1}=1-(k+1)=- k\le -2$ which contradicts convexity. By the algorithm, $p_{i-k,k+2}=1$. This, however, by checking the derivatives (at the boundary of the interval), that 
 $I^*_{k+2}= \emptyset$ as $3 \le k\le n-3$. Thus, there is no convex permutation matrix in this case. 
 
  {\em Case} $5$: $k=2$. Then by checking the possible derivatives at the boundary of the interval $I_s(P)$ for each $s$, one derives that $p_{i1}=p_{i+1,2}=p_{i-1,3}=1$ and then $p_{i+1,4}=1$, $p_{i-2,5}=1$ etc. The only possibility is then that $i=\lfloor n/2 \rfloor$ and we obtain the matrix $\Pi^*_n$.
 \end{proof}

 \begin{example}
 {\rm 
   The convex permutation matrices of order $n=6$ are the following 4 matrices 
   \[
   \left[
     \begin{array}{c|c|c|c|c|c}
       1&&&&& \\ \hline 
       &1&&&& \\ \hline 
       &&1&&& \\ \hline 
       &&&1&& \\ \hline 
       &&&&1& \\  \hline 
       &&&&&1      
     \end{array}
   \right], \; \;\;
   \left[
     \begin{array}{c|c|c|c|c|c}
       &&&&&1 \\ \hline 
       1&&&&& \\ \hline 
       &1&&&& \\ \hline 
       &&1&&& \\ \hline 
       &&&1&& \\  \hline 
       &&&&1&      
     \end{array}
   \right],  
   \]
  and  
    \[
   \left[
     \begin{array}{c|c|c|c|c|c}
       &&&&1& \\ \hline 
       1&&&&& \\ \hline 
       &1&&&& \\ \hline 
       &&1&&& \\ \hline 
       &&&1&& \\  \hline 
       &&&&&1      
     \end{array}
   \right], \;\;\;
   \Pi^*_6=\left[
     \begin{array}{c|c|c|c|c|c}
       &&&&1& \\ \hline 
       &&1&&& \\ \hline 
       1&&&&& \\ \hline 
       &1&&&& \\ \hline 
       &&&1&& \\  \hline 
       &&&&&1      
     \end{array}
   \right].
   \]
   and those additional 4 obtained by reordering rows in the opposite order.  
 }
\end{example} \endproof

 \section{Coda}
 \label{sec:coda}
 
 In this concluding section, we discuss 1-Costas permutations and some additional properties of permutations  involving their  derivatives.
 
 For permutations one may consider properties similar to  Lipschitz properties of  functions defined on the real line. 
 We say that a permutation $\pi$, and  the corresponding permutation matrix $P_{\pi}$,  is {\em $L$-Lipschitz} if 
\begin{equation}
 \label{eq:Lipschitz}
   | \pi_i -\pi_j | \le L|i-j| \;\;\;(i,j \le n).
\end{equation}
Since we only consider permutations, the only  values of interest here  are $L=1, 2, \ldots, n-1$. It is easy to verify that (\ref{eq:Lipschitz}) is equivalent to to the simplified condition that 
$ | \pi_{i+1} - \pi_i| \le L$ ($i < n$), or, equivalently, $\max_i |D(\pi)_i| \le L$. The only permutations that are $1$-Lipschitz are the identity and the anti-identity permutations. An interesting question is to characterize permutations that are $L$-Lipschitz, for a given $L$. We believe this can at least  be done for $L=2$. 

%
 %


  Checking if a given permutation of order $n$ has the $1$-Costas property is easily done: compute all the $(n-1)$ derivatives, requiring $(n-1)$ arithmetic operations, and then sort these number ($O(n \log n)$ operations suffice).

Every $n\times n$ permutation matrix $P$  with the $1$-Costas property may be constructed, starting with the zero matrix, by a simple algorithm which, however, may result in failure:

\begin{enumerate}
 \item Place a 1 in some position in the first row.
 \item for $i=2,3, \ldots, n$, 
 \begin{itemize}
    \item[(a)] determine the permitted positions in row $i$, i.e.,  those positions $(i,j)$ that are (i) not in a column already occupied by a 1, and (ii) not in any $3$-line in rows $\le i$ with two 1's, and (iii) not in any parallelogram in rows $\le i$ with three 1's,
   \item[(b)] choose, if possible, a permitted position in row $i$  and place a 1 there.
   \end{itemize}
\end{enumerate}
	If the algorithm does not stop before $n$ ones have been placed, the resulting matrix is a permutation matrix satisfying the $1$-Costas property. 		
	

Since the derivative of  an $L$-Lipschitz permutation  of order $n$ can have at most $2L$ values, 
an  $L$-Lipschitz permutations cannot have the  $1$-Costas property if $n$ is large enough.

%

The  table in Figure 1 below gives the values of $C_n^{(1)}$ for $n\le 10$. In the table, $F_n^{(1)}=100 \cdot C_n^{(1)}/n!$ is the fraction of the permutation matrices that have the 1-Costas property. 
A table of the number of Costas permutations of order $n$ for $n\le 29$  can be found   in \cite{SCH}.
	
\begin{figure}[ht]	
\begin{center}
\begin{tabular}{|r|r|r|c|} \hline
  $n$ & $n!$ & $C_n^{(1)}$ & $F_n^{(1)}$   \\ \hline
   1 & 1 & 1 & 100.0  \\ \hline
   2 & 2 & 2 & 100.0  \\ \hline
   3 & 6 & 4 & 66.7  \\ \hline
   4 & 24 & 12 & 50.0  \\ \hline
   5 & 120 & 44 & 36.7  \\ \hline
   6 & 720 & 176 & 24.4  \\ \hline
   7 & 5040 & 788 & 15.6  \\ \hline
   8 & 40320 & 3936 & 9.8  \\ \hline
   9 & 362880 & 23264 & 6.4  \\ \hline
   10 & 3628800 & 152112 & 4.2  \\ \hline
  \end{tabular}	
  \end{center}
  \caption{Permutations and the $1$-Costas property.}
\end{figure}	
	
\begin{question} Is $F_n^{(1)}$ a decreasing function of $n$? It is likely that  $\lim_{n\rightarrow \infty} F_n^{(1)}=0$.
  \end{question}

A  {\it centrosymmetric permutation} of order $n$ is a permutation
$\pi=(i_1,i_2,\ldots,i_n)$ such that $i_k+i_{n+1-k}=n+1$ for $k=1,2,\ldots,n$. The corresponding $n\times n$ {\it centrosymmetric permutation matrix} is characterized by the property that it is invariant under a 180 degree rotation. 
%
%
The permutation in Example \ref{ex:first} is  centrosymmetric and has a palindromic discrete derivative.

\begin{example}\label{ex:centro}{\rm
		Let $n=8$ and $\pi=(2,3,5,8,1,4,6,7)$ with corresponding permutation matrix
		\[P_{\pi}=
\left[\begin{array}{c|c|c|c|c|c|c|c}
&1&&&&&&\\ \hline
&&1&&&&&\\ \hline
&&&&1&&&\\ \hline
&&&&&&&1\\ \hline
1&&&&&&&\\ \hline
&&&1&&&&\\ \hline
&&&&&1&&\\ \hline
&&&&&&1&\end{array}\right].\]
Then $\pi$ and $P_{\pi} $ are centrosymmetric. The difference triangle is
\[\begin{array}{rrrrrrrrrrrrrrr}
2&&3&&5&&8&&1&&4&&6&&7\\ 
&1&&2&&3&&-7&&\cellcolor{blue!25}3&&\cellcolor{blue!25}2&&\cellcolor{blue!25}1&\\ 
&&3&&5&&-4&&\cellcolor{blue!25}-4&&\cellcolor{blue!25}5&&\cellcolor{blue!25}3&&\\ 
&&&6&&-2&&-1&&\cellcolor{blue!25}-2&&\cellcolor{blue!25}6&&&\\ 
&&&&-1&&1&&\cellcolor{blue!25}1&&\cellcolor{blue!25}-1&&&&\\ 
&&&&&2&&3&&\cellcolor{blue!25}2&&&&&\\ 
&&&&&&4&&\cellcolor{blue!25}4&&&&&&\\ 
&&&&&&&5&&&&&&&
\end{array}.\]
Thus except for those repeats (highlighted) which follow from		 the centrosymmetric property, the difference triangle has no other repeats in its rows.
Denoting $\pi$ as $(i_1,i_2, i_3, i_4,i_5,i_6,i_7,i_8)$, the entries in the center of the difference triangle are $9-2i_4, 9-2i_3,9-2i_2,9-2i_1$.}
\hfill{$\Box$}\end{example}

We define a {\it Costas-centrosymmetric permutation} to be a centrosymmetric permutation $(i_1,i_2,\ldots,i_n)$  whose difference table has no repeats other than those required by the centrosymmetric property. Note that the required repeats are of the form
$a-b=(n+1-b)-(n+1-a)$. Equivalently, we consider only the differences $\pi_j-\pi_i$   where  $1\le i < j\le n+1-i$. 
Example \ref{ex:centro} is a Costas-centrosymmetric permutation. 

The  Costas  permutation $(1,3,9,10,13,5,15,11,16,14,8,7,4,12,2,6)$ of order $n=16$ appears in \cite{Drakakisslides}.
Reversing the second half of this permutation gives the  Costas-centrosymmetric  permutation
$(1,3,9,10,13,5,15,11,6,2,12,4,7,8,14,16)$. In \cite{Drakakisslides} its ``anti-reflective symmetry'' is noted.

If we take the Costa-centrosymmetric permutation $(2,4,3,1,8,6,5,7)$ of order 8
with difference triangle
\[\begin{array}{rrrrrrrrrrrrrrr}
2&&4&&3&&1&&8&&6&&5&&7\\ 
&2&&-1&&-2&&7&&-2&&-1&&2&\\ 
&&1&&-3&&5&&5&&-3&&1&&\\ 
&&&-1&&4&&3&&4&&-1&&&\\ 
&&&&6&&2&&2&&6&&&&\\ 
&&&&&4&&1&&4&&&&&\\ 
&&&&&&3&&3&&&&&&\\ 
&&&&&&&5&&&&&&&
\end{array}\]
and reverse the last half to get $(2,4,3,1,7,5,6,8)$,
we do not get a  Costas permutation  since row 1 of its difference triangle already gives a repeat:
\[\begin{array}{rrrrrrrrrrrrrrr}
2&&4&&3&&1&&7&&5&&6&&8\\ 
&2&&1&&-2&&6&&-2&&1&&2&\\ 
&&1&&-3&&4&&4&&-1&&3&&
\end{array}.\]
This leads to the following question:
\begin{question}\label{qu:reverse}{\rm 
  Let $n=2m$ where $m$ is an even integer. Define a {\it Costas half- permutation} of order $m$ to be a sequence $(a_1,a_2,\ldots,a_m)$  such that $\{a_1,a_2,\ldots,a_m\}$ consists of one integer from each of the pairs $\{i,2m+1-i\}$ for $i=1,2,\ldots,m/2$ and each of the rows of its corresponding difference triangle does not have any repeats? $($Note that to construct a Costas half-permutation of order $m$ one has to choose an integer in each pair $\{i,2m-i\}$ and then order them in some way.$)$}
  \hfill{$\Box$}
  \end{question}

  \begin{example}\label{ex:half}{\rm
Let $m=6$. Then choosing one integer from each of the pairs $\{1,12\}, \{2,11\},\{3,10\},\{4,9\},\{5,8\}, \{6,7\}$, namely, $1,2,10,9,8,7$, we get a Costas half-permutation:
\[\begin{array}{rrrrrrrrrrr}
1&&8&&10&&9&&2&&7\\ 
&7&&2&&-1&&-7&&5&\\ 
&&9&&1&&-8&&-2&&\\ 
&&&8&&-6&&-3&&&\\ 
&&&&1&&-1&&&&\\ 
&&&&&6&&&&&\end{array}.\]
 }\hfill{$\Box$}
\end{example}
  
  A more general question is:
  \begin{question}\label{qu:subcostas}{\rm 
   Let  $m$ and $n$ be positive integers with $m\le n$. Define a {\it Costas $m$-subpermutation of order $n$} to be a sequence
  $a_1,a_2,\ldots,a_m$ of distinct integers taken from $\{1,2,\ldots,n\}$  whose  difference table does not contain a repeat in any row. For each positive integer $n$, let $\gamma_n$ be the largest integer $m$ such that 
  there is a Costas $m$-subpermutation of order $n$. Thus $\gamma_n\le n$ with equality if and only if there exists a Costas permutation of order $n$.
  Investigate this parameter $\gamma_n$. Find the best constant $c$ such that $\gamma_n\ge cn$ for all $n$.	  Clearly, if there exists a Costas permutation of order $m$, then for all $n\ge m$, $\gamma_n\ge m$.}
  \hfill{$\Box$}
  \end{question}
  
  Rather than applying the Costas property to a sequence of $m$ integers taken from $\{1,2,\ldots,n\}$, one can attach signs to a permutation.
   Let $\pi$ be a {\it signed permutation} of order $n$, that is,
$\pi=(i_i,i_2,\ldots,i_n)$ where $|\pi|=(|i_1|,|i_2|,\ldots,|i_n|)$ is a permutation of $\{1,2,\ldots,n\}$. Then $\pi$ is a {\it Costas-signed permutation}
provided the difference triangle of $ (i_i,i_2,\ldots,i_n)$ (note: not the difference triangle of $|\pi|$) does not have any repeats in its rows.
Allowing negative signs  makes it easier to satisfy the  Costas property of no repeats in a row.

\begin{example} {\it \rm
Let $n=4$ and $\pi=(2,4,-1,-3)$.
Then the difference triangle is
\[\begin{array}{rrrrrrr}
2&&4&&-1&&-3\\ 
&2&&-5&&-2&\\ 
&&-3&&-7&&\\ 
&&&-5&&&
\end{array},
\]	
so that this is a  Costas-signed permutation of order $4$.
}\hfill{$\Box$}
\end{example}

\begin{question}\label{qu:signedCostas}{\rm 
Does there exists a Costas-signed permutation of order $n$ for every positive integer $n$?
For each integer $n\ge 1$,  find a construction for a Costas-signed permutation of order $n$.}
\hfill{$\Box$}\end{question}

Finally we note that higher dimensional Costas permutations have been investigated; see e.g. \cite{JedwabLin} and the references therein.


\begin{thebibliography}{99}

\bibitem{Bona04}M. B\'ona, 
{\em Combinatorics of Permutations}, 
CRC Press, Inc. Boca Raton, 2004. 

\bibitem{BR91} R.A.~Brualdi, H.J.~Ryser,
\newblock {\em Combinatorial Matrix Theory},
\newblock Cambridge University Press, Cambridge, 1991.

\bibitem{Drakakis06} K.~Drakakis, A review of Costas arrays, {\it J.~Appl. Math.}, 2006,  Mar. 2006, Art. no. 26385.

\bibitem{Drakakisslides} K.~Drakakis, An introduction to Costas arrays:\\ http://www1.spms.ntu.edu.sg/~ccrg/documents/basicTalkSingapore.pdf.

\bibitem{Golomb} S.W.~Golomb and H.~Taylor, Constructions and properties of Costas arrays, {\it Proceedings of the IEEE}, 72 (9) (1984), 1143--1163.

\bibitem{Jedwab} J.~Jedwab and J.~Wodlinger, 
     Structural properties of Costas arrays,
   {\it Adv. Math. Commun.}, 8 (2014), 241-256.
   
   \bibitem{Jedwab2} J.~Jedwab and J.~Wodlinger, The deficiency of Costas arrays, {\it IEEE Trans. Inform. Theory}, 60 (2014), no. 12, 7947-7954.
   
  \bibitem{JedwabLin} J.~Jedwab  and L.~Yen, Costas cubes, {\it IEEE Trans. Inform. Theory}, 64 (2018), no. 4, part 2, 3144-3149.
  
  \bibitem{SCH} C.N. ~Swanson, B.~Correll Jr, and R.W.~Ho, 
  Enumeration of parallelograms in permutation matrices for improved bounds on the density of Costas arrays, {\it Electron. J. Combin}, 23 (2016), no. 1, 1.44, 14 pp.
  
\end{thebibliography}
\end{document}